\documentclass{amsart}
\usepackage{graphicx} 
\usepackage{color,amsmath,amssymb,amsthm,mathtools,graphicx}
\usepackage{libertine}
\usepackage[libertine]{newtxmath}
\usepackage{bbold}
\usepackage{epsfig,fancyhdr}
\usepackage{dsfont} 
\usepackage{graphicx,float}
\usepackage{graphicx}
\usepackage{epsfig,color,fancyhdr,setspace}
\usepackage{epstopdf}
\usepackage{nccmath}
\usepackage{wrapfig}
\usepackage{soul}
\usepackage{import}
\usepackage{lineno}
\usepackage{stackrel}
\usepackage[dvipsnames]{xcolor}
\usepackage[allcolors = blue,colorlinks]{hyperref}
\usepackage[abs]{overpic}
\usepackage[center]{caption}
\usepackage{subcaption}
\usepackage{tikz-cd}
\usepackage{enumitem}
\counterwithin{figure}{section}
\usepackage{tikz-cd}
\usepackage{wasysym}
\usepackage{marvosym}
\usepackage{extarrows}
\usepackage{amsmath}
\usepackage{graphicx,amsmath,mathtools,float}
\usepackage{setspace}
\graphicspath{{./images/}}
\usepackage{stackrel}
\usepackage{yfonts}
\usepackage[bottom]{footmisc}

\usepackage[T1]{fontenc}
\usepackage{babel}
\usepackage{enumitem, hyperref}\makeatletter
\def\namedlabel#1#2{\begingroup
    #2%
    \def\@currentlabel{#2}%
    \phantomsection\label{#1}\endgroup
}
\makeatother
\usepackage{type1cm} 
\usepackage{blindtext}
\usepackage{scalerel,stackengine}
\usepackage{color}
\usepackage{ mathrsfs }
\usepackage{tocbasic}
\newtheorem{theorem}{Theorem}[section]
\newtheorem{proposition}[theorem]{Proposition}
\newtheorem{lemma}[theorem]{Lemma}

\newtheorem{corollary}[theorem]{Corollary}
\newtheorem{conjecture}[theorem]{Conjecture}
\theoremstyle{definition}
\newtheorem{definition}[theorem]{Definition}
\newtheorem{example}[theorem]{Example}
\newtheorem{remark}[theorem]{Remark}

\newtheorem{notation}[theorem]{Notation}
\newcommand{\im}{\operatorname{im}}

\newcommand{\Z}{\mathbb{Z}}
\newcommand{\Q}{\mathbb{Q}}
\newcommand{\m}{_{(m)}}

\newcommand{\id}{\text{id}}
\newcommand{\inv}{^{-1}}

\newcommand{\allzero}{\color{gray!60} (0,0,0)}
\DeclareMathOperator{\rank}{rank}
\DeclareMathOperator{\frank}{free\,rank}
\DeclareMathOperator{\aut}{Aut}
\title{Low Dimensional Homology of the Yang-Baxter Operators Yielding the HOMFLYPT Polynomial}

\author{Anthony Christiana}
\address{Department of Mathematics, The George Washington University, Washington DC, USA}
\email{{\rm ajchristiana@gwmail.gwu.edu}}

\author{Ben Clingenpeel}
\address{Department of Mathematics, The George Washington University, Washington DC, USA}
\email{{\rm ben.clingenpeel@email.gwu.edu}}

\author{Huizheng Guo}
\address{Department of Mathematics, The George Washington University, Washington DC, USA}
\email{{\rm hguo30@gwu.edu}}

\author{Jinseok Oh}
\address{ Department of Mathematics, Kyungpook National University, Daegu, South Korea}
\email{{\rm jinseokoh@knu.ac.kr}}

\author{J\'{o}zef H. Przytycki}
\address{Department of Mathematics, The George Washington University, Washington DC, USA and \newline \indent Department of Mathematics, University of Gda\'{n}sk, Gda\'{n}sk, Poland}
\email{{\rm przytyck@gwu.edu}}

\author{Xiao Wang}

\address{Department of Mathematics, Jilin University, Changchun, China}
\email{{\rm wangxiaotop@jlu.edu.cn}}

\author{Hongdae Yun}
\address{ Department of Mathematics, Kyungpook National University, Daegu, South Korea}
\email{{\rm yyyj1234@knu.ac.kr}}

\begin{document}


\title{Low Dimensional Homology of the Yang-Baxter Operators Yielding the HOMFLYPT Polynomial}

\date{February 27, 2025}

\maketitle

\begin{abstract}
    In this paper, we analyze the homology of the Yang-Baxter Operators $R\m$ yielding the HOMFLYPT polynomial, reducing the computation of the $n$-th homology of $R_{(m)}$ for arbitrary $m$ to the computation of $n+1$ initial conditions. We then produce the explicit formulas for the third and fourth homology. 
\end{abstract}

\tableofcontents

\section{Introduction}

L. Faddeev coins the term ``Yang-Baxter equation" after the appearance of the equation in the independent work of Yang \cite{Yan} in 1967 and Baxter \cite{Bax} in 1972. The former was studying the many-body problem and the latter the partition function of the eight-vertex lattice model. It should be noted, however, that the idea was already implicit in the work of Hans Bethe who was studying  commuting transfer matrices in statistical mechanics, and factorizable $S$ matrices in field theory \cite{Jimb}.  

In his seminal paper \cite{Jon2} Vaughan Jones observed the potential of Yang-Baxter operators to contribute to the theory of knot/link invariants (see also Jones' letter to Joan Birman \cite{Jon1} and \cite{Jon3}).  In particular, the $2$ and $3$- (co)cycles from the homology theory of Yang-Baxter operators may lead to a natural construction of invariants for links and $2$-links---or more generally, $n$-(co)cycles---to analyze codimension 2 embeddings; compare \cite{CKS,PrRo}. More work remains in this direction, but since its introduction, the Yang-Baxter equation has served as a bridge between mathematics and physics, especially in low dimensional topology.

A homology theory for Yang-Baxter operators was introduced independently by Lebed \cite{Leb-1} and the fifth author \cite{Prz-6}. This homology theory generalizes the homology theory of quandles \cite{CJKLS} and the homology theory of set-theoretic Yang-Baxter operators \cite{CES}. In this article, we will focus on a family of Yang-Baxter operators leading to the HOMFLYPT polynomial \cite{FYHLMO, PT}\footnote{The Yang-Baxter operator on level $m$ leads to a concrete substitution of HOMFLYPT polynomial, and they are related to $sl_m$ representations.}; see \cite{Jon3, Tur} for the construction.

The paper is organized as follows: In this section, we recall the definitions of precubic modules, Yang-Baxter operators, and Yang-Baxter homology. Furthermore we provide the Yang-Baxter operators $R\m$ giving the HOMFLYPT polynomial and we stress the fact that the Boltzmann weights of these operators depend only on the ordering of the basis. Expanding on that, in Section 2, we study the properties of chain complexes and order-preserving maps defined on the bases of these complexes. 
In Section 3, we present filtrations of chain complexes which allow us to split our chain complexes as the direct sum of simpler ones.  This allows us to reduce the computation of the $n$-th homology of $R\m$ for arbitrary $m$ to the computation of $n+1$ initial conditions.  In Section 4, we prove the following  conjecture from \cite{PrWa2} about the third homology of the Yang-Baxter Operator leading to the HOMFLYPT polynomial.

\begin{conjecture} \label{mainconjecture} 
    $ H_3(R_{(m)}) = k^{\frac{m(8-3m+m^2)}{6}}\oplus (k/(1-y^2))^{\frac{(m^2-1)(5m-6)}{6}}\oplus (k/(1-y^4))^{m(m-1)} $
\end{conjecture}

We extend our calculation to the fourth homology. In Section 5, we provide some speculations and future plans.

\subsection{Background on Yang-Baxter homology}
\subsubsection{Precubic modules and their homology}

Let $k$ be a commutative ring with identity (e.g. $k= \mathbb Z$ or $\mathbb Z[y^2]$). Recall that 
a chain complex $\{C_n,\partial_n \}_{n\in \mathbb Z}$, concisely $(C_\bullet, \partial_\bullet)$, is a sequence of  $k$-modules (abelian groups for $k= \mathbb Z$) $C_n$, and homomorphisms $\partial_n: C_n \to C_{n-1}$ such that $\partial_{n} \partial_{n+1} =0$ for any  $n$. One defines the $n$-th homology of a chain complex as
$H_n(C_\bullet) = \ker \partial_n / \im \partial_{n + 1}$.

In this paper, we recall the homology theory for Yang-Baxter Operators using the language of precubic modules:  


\begin{definition}\label{precubic}

 A precubic module $\mathscr{ C}$ is a collection of $k$-modules $C_{n}$, $n \geq 0$, together with maps called face maps or face operators,
$$d_{i}^{\epsilon}: C_{n} \to C_{n-1},~ i=1,...,n, ~ \varepsilon = 0,1$$
such that
$$d_i^{\varepsilon}d_j^{\delta}= d_{j-1}^{\delta}d_i^{\varepsilon} \mbox{ for } 1 \leq i<j \leq n.$$
\end{definition}

From such a precubic module, we can construct the the chain complex $(C_n,\partial_n)$ by setting $\partial_n=\sum_{i=1}^n(-1)^i(d_i^{(0)}-d_i^{(1)})$. The chain complex consequently leads to a homology (and cohomology) theory.

In this paper, we use the notation $d^{\ell}$ and $d^r$ (i.e. $\epsilon\in \{l,r\}$) in agreement with the graphical interpretation of our chosen face maps (see Figure \ref{face map}).
The homology of Yang-Baxter operators, as developed in \cite{Leb-2, Prz-8} is defined as an interesting example of this general construction. In the next subsection, we construct precubic modules yielded by Yang-Baxter operators.

\subsubsection{Yang-Baxter operators}
Let us recall the definition of the Yang-Baxter equation and Yang-Baxter operators.

\begin{definition}\label{YBE}
Let $k$ be a commutative ring and $V$ be a $k$-module.  If a $k$-linear map, $R:$ $V\otimes V \to V\otimes V$, satisfies the following equation\\
$$(R\otimes Id_{V})\circ (Id_{V}\otimes R)\circ (R\otimes Id_{V})=(Id_{V}\otimes R)\circ (R\otimes Id_{V})\circ (Id_{V}\otimes R),$$\ \\
then we say $R$ is a pre-Yang-Baxter operator. The equation above is called a Yang-Baxter equation.  If, in addition, $R$ is invertible, then we say $R$ is a Yang-Baxter operator. 
\end{definition}
 From now on, we assume $k$ is a commutative ring with identity and $V=kX$ is a free $k$-module with a basis set $X$. Now we provide the family of Yang-Baxter operators yielding the HOMFLYPT polynomial, which is the main object of study in this paper.

\begin{definition}
 Let $X_{(m)}=\{v_{1},\dots ,v_{m}\}$ with the natural ordering $v_{i} < v_{j}$ if $i < j$, and $V_{(m)}=kX_{(m)}$. We recall the Yang-Baxter operator yielding the HOMFLYPT polynomial on level $m$; this is an operator $R_{(m)}$ \cite{Jon3,Tur}:

    $R_{(m)}:V_{(m)}\otimes V_{(m)} \rightarrow V_{(m)}\otimes V_{(m)}$ with coefficients (Boltzmann weights)  given by \[ R^{ab}_{cd} = \begin{cases}
        q & \text{if} \quad d = a = b = c; \\
        1 & \text{if} \quad a = d \neq b = c; \\
        q-q^{-1} & \text{if} \quad c = a < b = d; \\
        0 & \text{otherwise.}
    \end{cases} \]
This means for example that \[ R_{(2)} = \begin{bmatrix}
    q & 0 & 0 & 0 \\
    0 & q-q^{-1} & 1 & 0 \\
    0 & 1 & 0 & 0 \\
    0 & 0 & 0 & q
\end{bmatrix}. \] 

\end{definition}

Following \cite{Prz-8}, we define the homology of the Yang Baxter operator which is column unital. To have a chain complex, we modify the family of Yang-Baxter operators defined above by dividing the entries by the sum of each column with a change of variable $y^{2}=1/(1+q-q^{-1})$. The resulting operator is in fact a Yang-Baxter operator over the ring $k=Z[y^2]$ after this change of variables; see \cite{PrWa2}. The normalization is given formally in the following definition, and in the remainder of the paper we use $R\m$ to refer to this normalized operator. The fact that $k$ is not a PID complicates our concrete calculations.

\begin{definition}\label{def:YB-operator}
    Let $X_{(m)}=\{v_{1},\dots ,v_{m}\}$ with the natural ordering $v_{i} <  v_{j}$ if $i < j$, and $V_{(m)}=kX_{(m)}$. The normalized Yang-Baxter operator is the operator $$R_{(m)}:V_{(m)}\otimes V\m \rightarrow V\m\otimes V\m$$ with coefficients (Boltzmann weights) 
 given by \[ R^{ab}_{cd} = \begin{cases}
        1 & \text{if} \quad d = a \geq b = c; \\
        y^2 & \text{if} \quad d = a < b = c; \\
        1 - y^2 & \text{if} \quad c = a < b = d; \\
        0 & \text{otherwise.}
    \end{cases} \]
\end{definition} This means, for example, that \[ R_{(2)} = \begin{bmatrix}
    1 & 0 & 0 & 0 \\
    0 & 1 - y^2 & 1 & 0 \\
    0 & y^2 & 0 & 0 \\
    0 & 0 & 0 & 1
\end{bmatrix}. \] 

\begin{remark}

\begin{enumerate}
    \item The above family of operators are column unital matrices, e.g. stochastic matrices\footnote{A (column) stochastic matrix is a square matrix whose columns are probability vectors. A probability vector is a numerical vector whose entries are real numbers between 0 and 1 whose sum is 1.}. When $|y|\leq 1$, this family of Yang-Baxter operators are stochastic.
    \item These two families of Yang-Baxter operators are all leading to the HOMFLYPT polynomial, see \cite{PrWa2, Wan}.
\end{enumerate}

\end{remark}
 The following is a construction of precubic modules from column unital Yang-Baxter operators.

\begin{definition}
Let $C_{n}= V^{\otimes n}$. We illustrate the face maps $d_{i}^{l},$ $d_{i}^{r}$ and their difference through Figure \ref{face map}\footnote{We should also acknowledge stimulating observations by Ivan Dynnikov who was in audience of Przytycki's Moscow talks and asked sharp questions, May 30, 2002 (compare \cite{Prz-6}).  It was observed few weeks before by Victoria Lebed working on her Ph.D. thesis \cite{Leb-1,Leb-2}}. More precisely, whenever we see a crossing, we apply the Yang-Baxter operator, and for straight lines, we apply the identity map. When hitting the left wall, we delete the first tensor element of each basis in the linear combination, and when hitting the right wall, we delete the last tensor element of each basis in the linear combination.\footnote{Note that we draw the diagrams of these face maps with left and right walls, consistent with the work of Lebed who studies face maps with a more general wall interaction. Imagine that our basic elements are particles, and that the Yang-Baxter operation describes their collisions. In this viewpoint, our wall axiom, deletion, corresponds to the absorption of the particle by the wall. Lebed formalizes these interactions. } 
\end{definition}

\begin{figure}[H]
    \centering
    \includegraphics[width=.8 \linewidth]{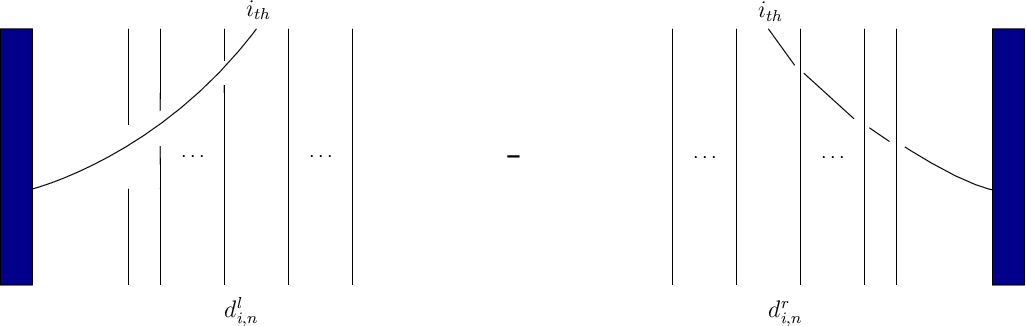}
\caption{Face map $d_{i,n}$.\label{face map}}
    \label{fig:enter-label}
\end{figure}

\begin{remark}

Note that $(C_{n},d_{i}^{\epsilon})$ with  $\epsilon \in \{l,r\}$ and  $1 \leq i \leq n$ is a precubic module for column unital Yang-Baxter  matrices \cite{Prz-8}. The chain complex from the precubic module $(C_{n},\partial_{n})$ is called the two-term pre-Yang-Baxter chain complex.  It's homology is called the two-term pre-Yang-Baxter homology $H_{n}(R).$

\end{remark} 

Here we demonstrate one concrete calculation of a face map.

\begin{example}

Let us consider the Yang-Baxter operator $R_{(m)}$ introduced in Definition \ref{def:YB-operator}. Let $a,b,c,d\in X_{(m)}$ be the basis elements with ordering $a\leq b\leq c\leq d$.  As shown in the left-hand side of Figure \ref{D4-compTree2-1}, 

\begin{align*}
    d_{4}^{l}(a\otimes b\otimes c\otimes d) &= (1-y^{2})d_{3}^{l} (a\otimes b\otimes c)\otimes d +y^{2}(d_{3}^{l}(a\otimes b\otimes d)\otimes c) \\
    &= (1-y^{2})[(1-y^{2})d_{2}^{l}(a\otimes b) \otimes c\otimes d + y^{2}d_{2}^{l}(a\otimes c) \otimes b\otimes d] \\
    &+ y^{2}[(1-y^{2})d_{2}^{l}(a\otimes b) \otimes d\otimes c + y^{2}d_{2}^{l}(a\otimes d) \otimes b\otimes c] \\
    &= (1-y^{2})^{3}(b\otimes c\otimes d)+y^{2}(1-y^{2})^{2}(a\otimes c\otimes d)\\
    &+y^{2}(1-y^{2})^{2}(c\otimes b\otimes d)+y^{4}(1-y^{2})(a\otimes b\otimes d) \\
    &+y^{2}(1-y^{2})^{2}(b\otimes d\otimes c) +y^{4}(1-y^{2})(a\otimes d\otimes c) \\
    &+y^{4}(1-y^{2})(d\otimes b\otimes c)+y^{6}(a\otimes b\otimes c).
\end{align*}

\end{example}

\begin{figure}[H]
\centering
\scalebox{.19}{\includegraphics{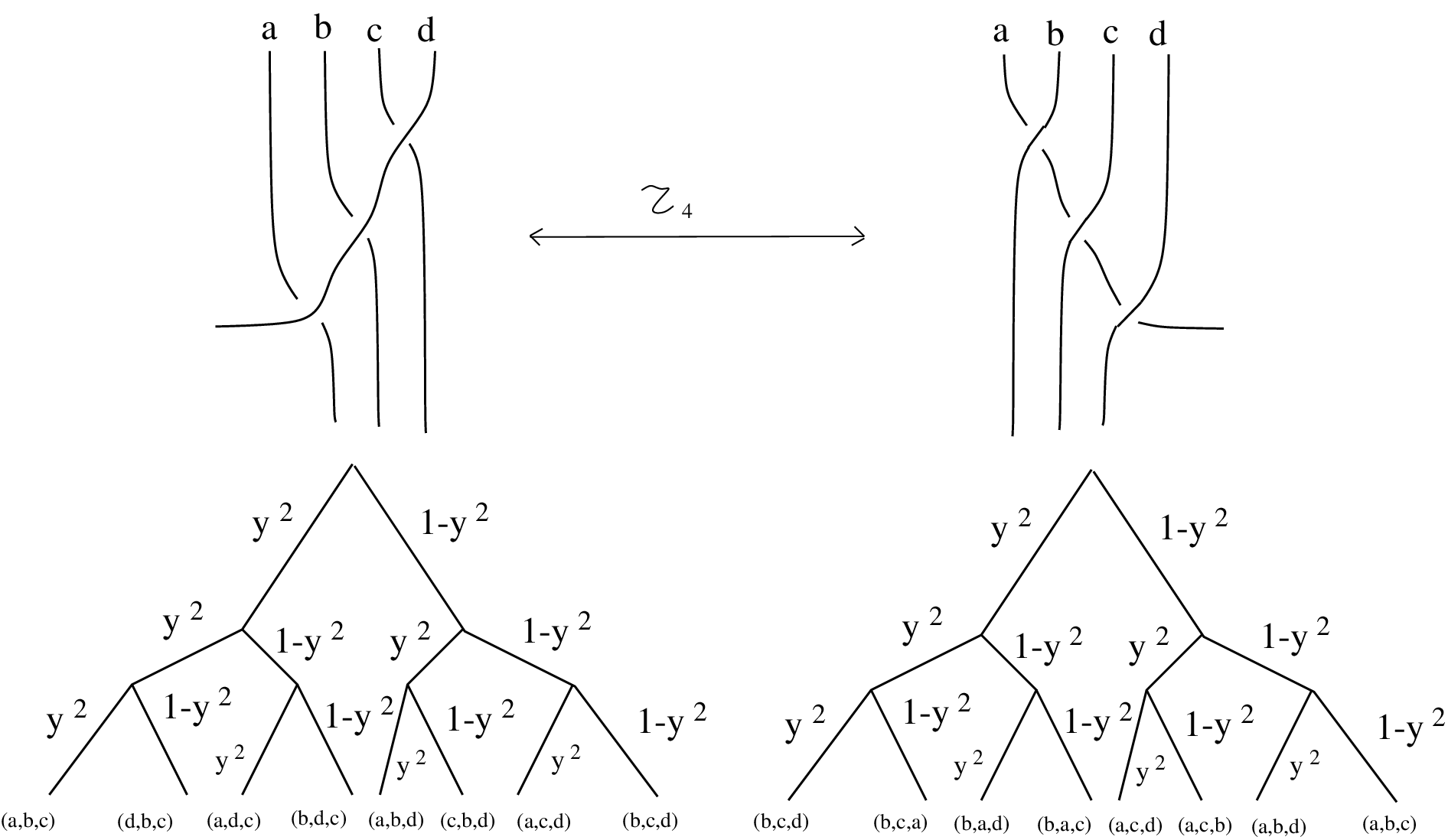}}
\caption{Computational tree for $d_4^{\ell}(a,b,c,d)$ where $a\leq b \leq c \leq d $ and its image under $\tau_4$}
\label{D4-compTree2-1}
\end{figure}

\section{Duality and ordering}\label{sec:DS}
\subsection{Duality and strictly order-preserving maps}
Functionally, the homology of $R_{(m)}$ is a homology theory of finite ordered sequences. The boundary map $\partial_n = \sum_{i=1}^{n} (-1)^{i}(d_i^l - d_i^r)$ has face maps $d_i^{\epsilon}$ defined in terms of the Yang-Baxter operator, $R:V \otimes V \to V \otimes V$. In turn, we note that $R(a,b) = \sum_{c,d} R^{a,b}_{c,d}(c,d)$ has Boltzmann weights depending \textit{only on the relative ordering of} $a,b \in X\m$. So we expect that maps which preserve the relative order of basis elements should accordingly preserve homology groups.

In this section, we put a finer point on this observation---that the relative ordering of basis elements is important for determining homology. First, we prove that maps between bases which strictly preserve order (including by completely reversing ordering) induce chain map isomorphisms. Then, we prove that weakly order-preserving maps yield chain maps.

\smallskip

The chain complex and homology of Yang-Baxter operators $R_{(m)}$ admits a natural involution, $(\tau_m)_{\#}: C^m_n \to C^m_n$ and   $(\tau_m)_* : H_n(R_{(m)})\to H_n(R_{(m)})$, $n\geq m$, defined on the level of chain modules $C^m_n= kX_{(m)}^n = (kX_{(m)})^{\otimes n}$. This involution has a nice graphical interpretation;  see Figure \ref{D4-compTree2-1}.

We will describe this involution step by step. For fixed $m$ and  $X_{(m)} = \{v_1, v_2, \dots, v_m \}$,  $\tau_m: X_{(m)} \to X_{(m)}$ is defined by\footnote{For simplicity we identify $v_i$ with $i$. Formally it should be $\tau_m(v_i)=v_{m+1-a}$.} $\tau_m(a)= m+1-a$. We extend $\tau_{m}$ to 
$X^n\m= \{(a_1,a_2,...,a_n)\in X^n \ | \ a_i \leq m \}$ by defining 
$$\tau_m(a_1,a_2,...,a_{n-1},a_n)=(\tau_m(a_n),\tau_m(a_{n-1}),...,\tau_m(a_2),\tau_m(a_1))=(m+1-a_n,...,m+1-a_1).$$
The important property of $\tau_m$ is that it is a chain map, or more generally, a morphism of precubic modules. That is:
\begin{proposition}\label{tau}
For $\tau$ as defined above, \\
\begin{enumerate}
\item[(1)] Since 
$$\tau_m(\partial_n(a_1,a_2,...,a_n))= (-1)^{n}\partial_n(\tau_m(a_1,...,a_n)),$$
\smallskip
the signed duality map $(-1)^{n-1}\tau_m$ is a chain map.

\smallskip

\item[(2)] We have a morphism of precubic modules that is
$$\tau_md_{i,n}^{\ell} = d^r_{n+1-i,n}\tau_m .$$
\end{enumerate}
\end{proposition}
\begin{proof} (1) follows directly from (2).\\
To show (2) we check that on the every level of the computational tree for $d^{\ell}_{i}(a_1,....,a_n)$ and $d^r_{n+1-i}(\tau_m(a_n),...,\tau_m(a_1))$, the involution $\tau_m$ (rotation by $y$-axis) is respected. Compare Figure \ref{D4-compTree2-1}.
\end{proof}

We remark that our involution comes from the involution on the braid group $B_n$ sending the generator $\sigma_i$ to $\sigma_{n-i}$. That is,
If $\gamma \in B_n$ and $\gamma = \sigma_{n_1}^{\varepsilon_1} \sigma_{n_2}^{\varepsilon_2} ... \sigma_{n_k}^{\varepsilon_k}$ then 
$\tau_n(\gamma)$ is obtained from $\gamma$ by a half-twist; that is
$\tau_n(\gamma)=\sigma_{n-n_k}^{\varepsilon_k}...\sigma_{n-n_2}^{\varepsilon_2} \sigma_{n-n_1}^{\varepsilon_1}$. If $F: B_n \to \aut (kX^n_{(m)})$ is a representation of the braid group $B_n$ which sends $\sigma_i$ to $R_{(m)}$ then $\tau_m F(\gamma)= F\tau_n(\gamma)$. 
\begin{example}
Assume $a<b<c$ and $m=3$; then we compute that
$$-\partial_3(a,c,b)= (1-y^2)\bigg((b,c)-(a,c)+y^2((c,b)-(c,a))\bigg).$$
Then, by duality, by applying $\tau_3$ to both sides of the equation, we get:
$$-\partial_3(b,a,c)= (1-y^2)\bigg(-(a,b)+(a,c)+y^2(-(b,a)+(c,a))\bigg),$$ 
which of course agrees with the direct calculation.
\end{example}
\begin{remark}\label{Involutiontau}
We can graphically interpret the involution $\tau_m$ as a rotation about the $y$-axis (see Figure \ref{D4-compTree2-1}). As often happens, we can complement rotation by a shift along the $x$-axis. 
\end{remark}

\begin{proposition}\label{Shift}
\begin{enumerate}
\item[(1)] (Shift)\\ Let $X_{m,m}=X_{(m)}- \{m\}$ (the largest letter is not used) and $X_{m,1}=X_{(m)}- \{1\}$ (the smallest letter is not used).                                     Let $s_{m,n}: kX_{m,m} \to kX_{m,1}$ given by $s_{m,n}((a_1,a_2,...,a_n)= (a_1+1,a_2+1,...,a_n+1)$. Then $s_{m,*}$ induces the precubic module isomorphism
and is therefore a chain map isomorphism.                                                                                                                                                 \item[(2)] (Filling or creating a gap)\\ Let $X_{m,k}=X_{(m)}- \{k\}$ (the set $X_{(m)}$ with the gap on level $k$).                                          Assume that $k<m$ and consider the map $s_k=s_{k,n}: kX_{m,k} \to kX_{m,k-1}$ given by                                                                                            $s_k(a_1,a_2,...,a_n)=(a'_1,a'_2,...,a'_n)$ with                                 \[ a'_i := \left\{      \begin{array}{ll}   a_i & \mbox{when $a_i > k$}, \\               a_i+1 & \mbox{when $a_i < k$}.
\end{array}                                                                              \right.                                                                                  \]                                                          Then $s_{k-1,n}(d_{i,n}^{\delta}(a_1,a_2,...,a_n)) = d_{i,n}^{\delta}s_{k,n}(a_1,a_2,...,a_n))$,
where $\delta \in \{\ell, r\}$. In particular, $s_*$ is an isomorphism of precubic modules so also an isomorphism on chain complexes.
\end{enumerate}
\end{proposition}                                                                                                      \begin{proof} (1) follows directly from (2) and (2) follows from the property of the computational tree which is, up to shifting, the same for both $d_{i,n}^{\delta}(a_1,a_2,...,a_n)$ and also $d_{i,n}^{\delta}s_{k,n}(a_1,a_2,...,a_n)$.
\end{proof}          
Shifts and filling/creating gaps are special cases of order preserving maps, which play a crucial role in our paper and will be discussed in the following subsection.  In fact, strictly order-preserving maps on bases of equal cardinality can be seen as compositions of the filling/creating gap maps, and therefore induce chain map isomorphisms.


\begin{corollary}
Let $S \subset X_{(m)}$ and $T \subset X_{(m')}$ for $m, m' \in \mathbb{Z}$ and $|S| = |T|$. If $f:S \to T$ is strictly order-preserving, then $f_{\#}: (k S)_\bullet \to (k T)_{\bullet}$ is a chain map isomorphism.
\end{corollary}

\begin{proof}
We naturally identify $X_{(m)}, X_{(m')}$ as ordered subsets of $X_{(\max\{m,m'\})}$. Then any strictly order-preserving map $f:S \to T$ can be regarded as the  composition of maps from (1) and (2) in \ref{Shift}, and thus, induces a chain map isomorphism.
\end{proof}

\subsection{Weakly order-preserving maps are chain maps}

In fact, more is true, in that even weakly order-preserving maps $X\m \to X_{(m')}$ yield chain maps. Since such maps need not preserve Boltzmann weights, we make essential use of the column-unital condition. 

\begin{lemma}\label{lem:weak-order}
    Let $R\m$ and $R_{(m')}$ be Yang-Baxter operators as  in Definition \ref{def:YB-operator} on $kX\m$ and $kX_{(m')}$, respectively. Then a weakly order-preserving map $f: X\m \to X_{(m')}$ satisfies the condition \[ R_{(m')} \circ (f \otimes f) = (f \otimes f) \circ R\m. \tag{$1$} \label{morph-cond} \] 
\end{lemma}

\begin{proof}
    Applying both of the above to $(a,b)$, we get \begin{align*} 
        (R_{(m')} \circ (f \otimes f))(a,b) &= \sum_{(c,d)} R^{ab}_{cd} \cdot (f(c),f(d)) \\
        ((f \otimes f) \circ R\m)(a,b) &= \sum_{(x,y)} R^{f(a)\,f(b)}_{xy} \cdot (x,y).
    \end{align*} The Boltzmann weights are nonzero only if $\{ c,d \} = \{a,b\}$ for the first sum and $\{x,y\} = \{f(a),f(b)\}$ for the second sum. If $f(a) = f(b)$, then the only tuple involved is $(f(a), f(a))$, and so the coefficients of both sums above add to 1 by the column-unital condition. If instead $f(a) \neq f(b)$, then $f$ preserves order strictly, and hence preserves the Boltzmann weights; thus, the two sums are again the same. 
\end{proof}

\begin{proposition}\label{prop:morph-chain}
    Any map $f: X\m \to X_{(m')}$ satisfying Equation \eqref{morph-cond} yields a chain map $f_\sharp$ by $f_\sharp = f^{\otimes n}: C_\bullet^m \to C_\bullet^{m'}$. 
\end{proposition}

\begin{proof}
    The left and right boundary maps can be defined as 

    \begin{align*}
        d_i^\ell &= L_\text{wall} \circ R \otimes \id^{\otimes (n - 2)} \circ \cdots \circ \id^{\otimes (i - 2)} \otimes R \otimes \id^{\otimes (n - i)} \\
        d_i^r &= R_\text{wall} \circ \id^{\otimes (n - 2)} \otimes R \circ \cdots \circ \id^{\otimes (i - 1)} \otimes R \otimes \id^{\otimes (n - i - 1)}. 
    \end{align*} where we use $L_{\text{wall}}$ to mean the interaction with the left wall in which the leftmost element of the tensor basis is dropped, and similarly for $R_{\text{wall}}$. Then $f \otimes f$ commuting with $R$ implies that $f^{\otimes n}$ commutes with all the $\id^{\otimes k} \otimes R \otimes \id^{\otimes j}$ terms, and $f^{\otimes n}$ also commutes with dropping a right or left tensor factor in each basis element. 
\end{proof}

Combining Lemma~\ref{lem:weak-order} and Proposition~\ref{prop:morph-chain}, we have that any weakly order-preserving $f: X\m \to X_{(m')}$ yields a chain map $f_\sharp: C_\bullet^m \to C_{\bullet}^{m'}$.

\begin{corollary}\label{cor:split-mono}
    Any strictly order-preserving map $f: X\m \to X_{(m')}$ induces a split monomorphism $f_\sharp: C_\bullet^m \to C_\bullet^{m'}$ of chain complexes. It then follows that $f$ induces a split monomorphism $f_*$ on homology as well. 
\end{corollary}

\begin{proof}
    We just need to construct a weakly order-preserving left-inverse. In general this is not unique---for letters in $\im f$, there is no choice, but the other letters in $X_{(m')}$ can be sent anywhere, with the only restriction being that the resulting map $g: X_{(m')} \to X_{(m)}$ be weakly order-preserving. For example, the map $f: X_{(3)} \to X_{(4)}$ given below has left-inverses $g$ and $g'$.
    \begin{align*}
        a &\mapsto x \qquad \qquad \qquad \,\,\, a \mapsto x \\
        x \mapsto a \qquad \qquad \qquad \, b &\mapsto x \qquad \qquad \qquad \,\,\, b \mapsto y \\
        f : \quad y \mapsto c \qquad \qquad g: \quad c &\mapsto y \qquad \qquad g': \quad c \mapsto y \\
        z \mapsto d \qquad \qquad \qquad \, d &\mapsto z \qquad \qquad \qquad \,\,\, d \mapsto z
    \end{align*} Since these left-inverses are weakly order-preserving, they yield chain maps by the above. Since $g_\sharp f_\sharp = \id_{C_\bullet^{m}}$, $f_\sharp$ is a split monomorphism. 
\end{proof}

\begin{remark}
    In Section 3 we will use these left inverses to provide splittings of certain chain complexes, but there we will require additional assumptions to ensure the maps they induce are well-defined. 
\end{remark}

\section{Filtrations and splittings}
\label{sec:filtrations-splittings}

\subsection{Filtration by alphabet size}
\label{subsec:letters-used}
With the foundation laid out, we now set about to computing $H_n(R_{(m)}) : = H_n(C_{\bullet}^m).$ We do so by first decomposing $C_{\bullet}^m$ via succesive filtrations until we arrive at a presentation of $H_n(R_{(m)})$ as a direct sum of homology groups which are more easily computed. For a fixed $n$ and arbitrary $m$, we will obtain a decomposition of $C_\bullet^m$ into only $n + 1$ simpler pieces. This is the decomposition given in Theorem~\ref{DirectSumResult}.


The first step in this process uses the observation that $X_{(m - 1)}$ sits inside $X\m$. This gives us a filtration \[ C_\bullet^1 \subseteq C_\bullet^2 \subseteq \cdots \subseteq C_\bullet^{m}.  \] of chain complexes. From this filtration, we obtain for each $1 < p \leq m$ the following short exact sequence of chain complexes: \[ \begin{tikzcd}
    0 \arrow[r] & C_\bullet^{p - 1} \arrow[r, "\alpha"] & C_\bullet^p \arrow[r, "\beta"] & C_\bullet^p / C_\bullet^{p - 1} \arrow[r] & 0
\end{tikzcd} \tag{1} \label{intl-ses} \] where $\alpha$ is the inclusion and $\beta$ is the standard quotient map.

\begin{proposition} 
    The chain complex associated with this filtration is isomorphic to the associated graded chain complex: \[
        C_\bullet^{m} \cong \bigoplus_{p=2}^{m} C_\bullet^{p} / C_\bullet^{p-1}.
    \]
    It follows that:
    \[
        H_n(C_\bullet^{m}) \cong \bigoplus_{p=2}^{m} H_n(C_\bullet^{p} / C_\bullet^{p-1}). 
    \]
\end{proposition}

\begin{proof}  
    Since the inclusion \( \alpha : C_\bullet^{p-1} \to C_\bullet^p \) in the short exact sequence \eqref{intl-ses} comes from the strictly order-preserving map $X_{(p - 1)} \to X_{(p)}$, the induced map $\alpha$ is a split monomorphism of chain complexes (see Corollary~\ref{cor:split-mono}). The claimed decomposition follows by induction on $m$.
\end{proof}

\subsection{Further reductions by filtrations}
\label{subsec:final-complex}

In the proof above, we have seen that the short exact sequence \[ \begin{tikzcd}
    0 \arrow[r] & C_\bullet^{m - 1} \arrow[r, "\alpha"] & C_\bullet^m \arrow[r, "\beta"] & C_\bullet^m / C_\bullet^{m - 1} \arrow[r] & 0
\end{tikzcd} \] of complexes gives rise to the splitting $C_\bullet^m \cong C_\bullet^{m - 1} \oplus C_\bullet^m / C_\bullet^{m - 1}$, and that we can use this inductively to further decompose the left factor. We turn our attention now to the right factor, first noting that the chain groups of the quotient complex $C_\bullet^m / C_\bullet^{m - 1}$ have basis elements consisting exactly of those (equivalence classes with representative) tuples that use the top letter $v_m$ at least once. We will therefore denote this chain complex by $C_\bullet^{m,1,m-1}$ to say that the basis elements must use the top letter at least once, and that otherwise all $m - 1$ bottom letters may be used. In general, we introduce the following notation: 


\begin{definition}\label{chain-module-def}
    We define $C_\bullet^{m,u,\ell}$ for $\ell \leq m - u$ inductively on $u\geq 0$. First, let $C_\bullet^{m,0,m} = C_\bullet^m$. Assuming  $C_\bullet^{m,u,m-u}
    $ has been defined, we take $C_\bullet^{m,u+1,m-u-1}$ to be the quotient chain complex of $C_\bullet^{m,u,m-u}$ by $C_\bullet^{m,u,m-u-1}$, where $C_\bullet^{m,u,m-u-1}$ is the sub chain complex of $C_\bullet^{m,u,m-u}$ whose chain groups have basis tuples that do not use the letter $v_{m-u}$ and use all $u$ top letters. Next, we define $C_\bullet^{m,u,\ell}$ for $\ell < m - u$ to be the sub chain complex of $C_\bullet^{m,u,m-u}$ without using $v_{\ell+1}, \dots , v_{m-u}$ in the basis tuples (but still using all $u$ top letters). We interpret the chain module $C_n^{m,u,\ell}$ as being the free $k$-module generated by those tuples that use the top $u$ letters at least once, and otherwise use only the bottom $\ell$ letters. This inductive procedure is given in more careful detail below (compare to the short exact sequence \eqref{ses-full}). We adopt the convention that $C_\bullet^{m,u} = C_\bullet^{m,u,m-u}$. 
\end{definition}

Using this notation, we can rewrite the short exact sequence we considered in \ref{subsec:letters-used} as follows: \[ \begin{tikzcd}
    0 \arrow[r] & C_\bullet^{m, 0 , m - 1} \arrow[r, "\alpha"] & C_\bullet^{m, 0 , m} \arrow[r, "\beta"] & C_\bullet^{m, 1, m - 1}  \arrow[r] & 0
\end{tikzcd} \] (using the observation that $C_\bullet^{m,0,m - 1}$ is isomorphic to the chain complex $C_\bullet^{m - 1, 0, m - 1} = C_\bullet^{m - 1}$). To continue, we note that $C_n^{m,1,m-1}$ includes the free $k$-module with basis consisting of those tuples that use the top letter $v_m$ at least once, but otherwise use only the first $m - 2$ bottom letters $v_1, v_2, \dots, v_{m - 2}$. That is, $C_n^{m, 1, m - 2} \subsetneq C_n^{m, 1, m - 1}$. Continuing, we see that \[ C_n^{m,1,1} \subsetneq C_n^{m,1,2} \subsetneq \cdots \subsetneq C_n^{m,1,m - 2} \subsetneq C_n^{m,1,m - 1}. \] This holds for all $n$ and gives a filtration of the complex $C_\bullet^{m,1,m-1}$, since anything in $\partial C_n^{m,1,i}$ is a $k$-linear combination of $(n - 1)$-tuples on the letters $\{v_1, \dots, v_i, v_m\}$, and all tuples must involve the top letter $v_m$ at least once, since otherwise they belong in $C_{n - 1}^{m, 0, m - 1}$ and are therefore zero in $C_{n - 1}^{m,1,i} \subset C_{n - 1}^{m,1,m - 1} = C_{n - 1}^{m,0,m} / C_{n - 1}^{m, 0, m - 1}$. 

This means in particular that we have the following short exact sequence \[ \begin{tikzcd}
    0 \arrow[r] & C_\bullet^{m, 1, m - 2} \arrow[r, "\alpha"] & C_\bullet^{m, 1, m - 1} \arrow[r, "\beta"] & C_\bullet^{m, 1, m - 1} / C_\bullet^{m, 1, m - 2} \arrow[r] & 0.
\end{tikzcd} \] Thinking about what is left of $C_n^{m, 1, m - 1}$ when taking the quotient by $C_n^{m,1,m-2}$, we see that the quotient is $C_n^{m, 2, m - 2}$, whose basis consists of $n$-tuples on $m$ letters that must use the top two letters $v_m$ and $v_{m - 1}$ at least once. Note that this is the second step of the inductive process in Definition~\ref{chain-module-def}.

Continuing this procedure, we see that for each $u < m - 1$ and $1 < \ell \leq m - u$, we have a filtration \[ C_n^{m,u,1} \subsetneq C_n^{m,u,2} \cdots \subsetneq C_n^{m, u, \ell - 1}\subsetneq C_n^{m,u,\ell}. \] Looking again at just the top two terms of these filtrations, we get short exact sequences \[ \begin{tikzcd}
    0 \arrow[r] & C_\bullet^{m, u, \ell - 1} \arrow[r, "\alpha"] & C_\bullet^{m, u, \ell} \arrow[r, "\beta"] & C_\bullet^{m, u + 1, \ell - 1} \arrow[r] & 0 \tag{2} \label{ses-full}
\end{tikzcd} \] each of which leads to a long exact sequence in homology.

At this point it becomes more convenient to discuss the chain groups $C_\bullet^{m,u,\ell}$ slightly differently:

\begin{remark}\label{rmk:notational-isom}
    If $u + \ell < m$, the alphabet for $C_\bullet^{m,u,\ell}$ has $m$ total letters, but to form tuples, we can draw only from the $u + \ell$ letters in $\{v_1, \dots, v_\ell, v_{m - u + 1}, \dots, v_m\}$. For this reason, we always have $C_\bullet^{m,u,\ell} \cong C_\bullet^{\ell + u, u, \ell}$. This allows us to think of $C_\bullet^{\ell + u, u, \ell}$ as a subcomplex of $C_\bullet^{m, u, m - u}$ via the strictly order-preserving map taking $v_i$ to itself when $i \leq \ell$ and taking $v_i$ to $v_{(m - u) + (i - \ell)}$ when $i > \ell$. 
\end{remark}

With this remark, we can rewrite the filtration above up to isomorphism as 

\[ C_n^{u + 1,u,1} \subsetneq C_n^{u + 2,u,2} \cdots \subsetneq C_n^{m - 1, u, m - u - 1}\subsetneq C_n^{m,u,m-u}, \] and then we have the short exact sequences \[ \begin{tikzcd}
    0 \arrow[r] & C_\bullet^{m - 1, u, m - u - 1} \arrow[r, "\alpha"] & C_\bullet^{m, u, m - u} \arrow[r, "\beta"] & C_\bullet^{m, u + 1, m - u - 1} \arrow[r] & 0
\end{tikzcd} \] from comparing the top two terms of the filtration. The advantage of this viewpoint is that we  can talk about everything now in terms of alphabets consisting only of consecutive letters. It also allows us to simplify the notation, because the third parameter $\ell = m - u$ is redundant. We can then write the short exact sequence as \[ \begin{tikzcd}
    0 \arrow[r] & C_\bullet^{m - 1, u} \arrow[r, "\alpha"] & C_\bullet^{m, u} \arrow[r, "\beta"] & C_\bullet^{m, u + 1} \arrow[r] & 0.
\end{tikzcd} \tag{3} \label{u-complex} \]

\begin{proposition}\label{prop:splitting}
    The short exact sequences of chain modules \eqref{u-complex} are split, giving $C_\bullet^{m,u} \cong C_\bullet^{m-1,u} \oplus C_\bullet^{m,u+1}$ as chain complexes for $u < m - 1$. 
\end{proposition}

\begin{proof}
    The map $\alpha: C_\bullet^{m - 1, u} \to C_\bullet^{m,u}$ is induced by the map $f: \{v_1, \dots, v_{m - 1}\} \to \{v_1, \dots, v_m\}$ sending $v_{i}$ to $v_i$ for $i < m - u$ and to $v_{i+1}$ for $i \geq m - u$. Since this $f$ is strictly order-preserving, $\alpha$ as a map from $C_\bullet^{m - 1}$ to $C_\bullet^{m}$ has left-inverses that are also chain maps by Corollary~\ref{cor:split-mono}. Not all of these are well-defined on the quotient complexes, so we choose the left-inverse $g$ of $f$ given by the following: \[ g(v_i) = \begin{cases}
        v_i & i < m - u \\
        v_{i - 1} & i \geq m - u.
    \end{cases} \] Inductively we assume $g$ induces well-defined chain maps $\overline{\alpha}_{i}: C_\bullet^{m,i,m-i} \to C_\bullet^{m-1,i,m-1-i}$ for all $i\leq u-1$, with the base case $i = 0$ following from the results of Section~\ref{sec:DS} since $g$ is weakly order-preserving. Then $g$ also induces a well-defined chain map $\overline{\alpha}_u: C_\bullet^{m,u} \to C_\bullet^{m - 1, u}$ since $C_\bullet^{m,u} = C_\bullet^{m, u - 1, m - u + 1} / C_\bullet^{m, u - 1, m - u}$ by definition and $\overline{\alpha}_u|_{C_\bullet^{m,u - 1, m - u + 1}} = \overline{\alpha}_{u - 1}$ with $\overline{\alpha}_u(C_\bullet^{m, u - 1, m - u}) \subset C_\bullet^{m - 1, u - 1, m - u - 1}$. Thus $\overline{\alpha} = \overline{\alpha}_u$ is a chain map with $\overline{\alpha}\alpha = g_\sharp f_\sharp = \id_{C_\bullet^{m - 1, u}}$. 
\end{proof}

\color{black}

We now have a way of splitting all the chain complexes $C_\bullet^{m,u}$ into simpler pieces. Because the short exact sequence \eqref{u-complex} used in Proposition~\ref{prop:splitting} requires $u < m - 1$, the final time we are able to make use of this splitting is with $C_\bullet^{m,m - 2} \cong C_\bullet^{m - 1, m - 2} \oplus C_\bullet^{m,m - 1}$. Since these two pieces serve as basic building blocks, we introduce special notation for them.

\begin{notation}\label{notat:final}
    We let $C_\bullet^{mf}$ denote the complexes that arise at the final inductive step. That is, $C_\bullet^{mf} = C_\bullet^{m,m-1}$. 
\end{notation}

With this notation, the final decomposition is written $C_\bullet^{m, m - 2} \cong C_\bullet^{(m - 1)f} \oplus C_\bullet^{mf}$. 

\begin{remark}
    It would appear more natural to define the `final' complex to be $C_\bullet^{m,m}$ where basis tuples must use {\it all} $m$ letters at least once. Although the chain modules split, there is no splitting of chain {\it complexes}, as the following counterexample shows. We have that $H_2(C_\bullet^{2,1}) \cong k \oplus k/(1 - y^2) \oplus k/(1 - y^4)$, whereas $H_2(C_\bullet^{1,1}) \cong k$ and $H_2(C_\bullet^{2,2}) \cong k \oplus k/(1 - y^2)$. The proof of Proposition~\ref{prop:splitting} does not apply to prove this supposed splitting, because it depends on the existence of a left-inverse that gives a bijection between the $u$ top letters of the alphabets of $C_\bullet^{m - 1, u}$ and $C_\bullet^{m, u}$, which is not possible if $u = m - 1$. In the example given here, the map $g: \{v_1, v_2\} \to \{v_1\}$ inducing $g_\sharp: C_\bullet^{2,1} \to C_\bullet^{1,1}$ takes $v_1 \mapsto v_1$. In the first alphabet, $v_1$ is not a top letter, but in the second, it is. 
\end{remark}

We conclude this section with the main technical result of the paper. 

\begin{theorem}\label{DirectSumResult}
    The chain complex $C_\bullet^{m,u}$ has the following direct sum decomposition: \[ C_\bullet^{m,u} \cong \bigoplus_{j = u + 1}^m \binom{m - u - 1}{j - u - 1} C_\bullet^{jf} \] for $u < m - 1$. In particular, $u = 0$ gives \[ C_\bullet^m \cong \bigoplus_{j = 1}^m \binom{m - 1}{j - 1} C_\bullet^{jf}. \]
\end{theorem}

\begin{proof}
    The proof is essentially contained in Figure~\ref{CompFiltr}. More formally, we argue by induction on $m - u$. For the base case when $m - u = 1$, the claim is just that $C_\bullet^{m,m - 1} = \binom{0}{0} C_\bullet^{mf} = C_\bullet^{mf}$, which is true by definition of the notation for final complexes. Assume now that the claim holds when $m - u = \ell$, and suppose $m - u = \ell + 1$. By Proposition~\ref{prop:splitting}, we have that $C_\bullet^{m,u} \cong C_\bullet^{m-1,u} \oplus C_\bullet^{m,u+1}$, and we can apply the induction hypothesis to both terms: \begin{align*}
        C_\bullet^{m,u} &\cong \bigoplus_{j = u + 1}^{m - 1} \binom{m - u - 2}{j - u - 1} C_\bullet^{jf} \oplus \bigoplus_{j = u + 2}^m \binom{m - u - 2}{j - u - 2} C_\bullet^{jf} \\
        &= C_\bullet^{(u + 1)f} \oplus \bigoplus_{j = u + 2}^{m - 1} \left( \binom{m - u - 2}{j - u - 1} + \binom{m - u - 2}{j - u - 2} \right) C_\bullet^{jf} \oplus C_\bullet^{mf} \\
        &= C_\bullet^{(u + 1)f} \oplus \bigoplus_{j = u + 2}^{m - 1} \binom{m - u - 1}{j - u - 1} C_\bullet^{jf} \oplus C_\bullet^{mf} = \bigoplus_{j = u + 1}^m \binom{m - u - 1}{j - u - 1} C_\bullet^{jf},
    \end{align*} as required. 
\end{proof}

\begin{figure}[H]
\centerline{\psfig{file=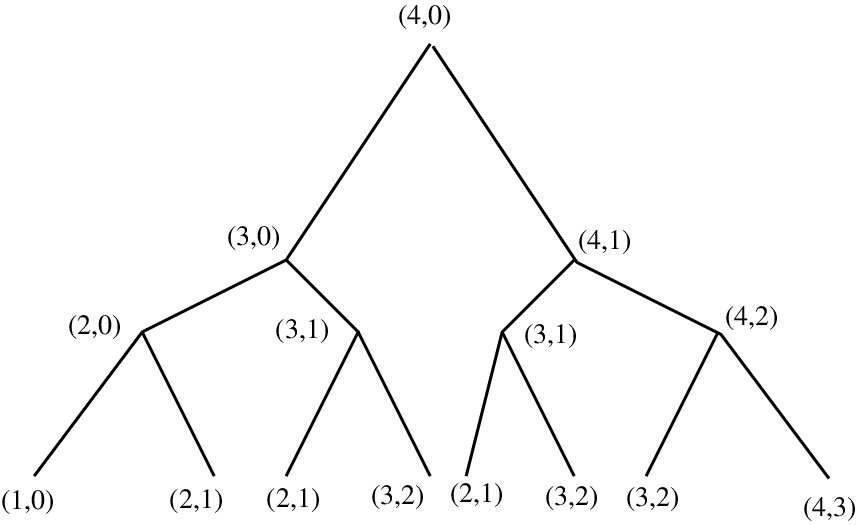,width=3.0in}}
\vspace*{5pt}
\caption{Illustration of the fact that $C_\bullet^4 \cong C_\bullet^{1,0} \oplus 3C_\bullet^{2,1} \oplus 3C_\bullet^{3,2} \oplus C_\bullet^{4,3}$, where e.g. $C_\bullet^{3,1}$ is represented by (3,1) in the tree. Note the paths ending at $(j, j - 1)$ go left $m - j$ times and right $j - 1$ times, so there are $\binom{m - 1}{j - 1}$ such paths, as claimed.}
\label{CompFiltr}
\end{figure}

The usefulness of this decomposition is that even though the number of terms depends on $m$, for $j > n + 1$, the chain module $C_n^{jf}$ has basis consisting of those $n$-tuples that use all $j - 1$ top letters, meaning the chain modules are zero after $j = n + 1$. Therefore the homology groups $H_n(C_\bullet^{jf})$ are also zero for $j > n + 1$. This means to compute the third and fourth homology for general $m$, we need only compute $H_3(C_\bullet^{jf})$ for $j \leq 4$ and $H_4(C_\bullet^{jf})$ for $j \leq 5$: \begin{align*} 
    H_3(C_\bullet^m) &\cong H_3(C_\bullet^{1f}) \oplus H_3(C_\bullet^{2f})^{\oplus (m - 1)} \oplus H_3(C_\bullet^{3f})^{\oplus \binom{m - 1}{2}} \oplus H_3(C_\bullet^{4f})^{\oplus \binom{m - 1}{3}} \tag{4} \label{3rd-decomp} \\
    H_4(C_\bullet^m) &\cong H_4(C_\bullet^{1f}) \oplus H_4(C_\bullet^{2f})^{\oplus (m - 1)} \oplus H_4(C_\bullet^{3f})^{\oplus \binom{m - 1}{2}} \\
    &\oplus H_4(C_\bullet^{4f})^{\oplus \binom{m - 1}{3}} \oplus H_4(C_\bullet^{5f})^{\oplus \binom{m - 1}{4}}. \tag{5} \label{4th-decomp}
\end{align*}

\section{Computation of $H_3$ and $H_4$ }\label{H3computation}

\subsection{Computation} \label{Computation}

The main result of this paper, Theorem \ref{DirectSumResult}, expresses $H_n(C_{\bullet}^{m})$ as the direct sum of homology groups of type $H_n(C^{j,j-1,1}_{\bullet})$ for $1 \leq j \leq  n+1$. In this section, we prove Conjecture~\ref{mainconjecture} on $H_3(C_{\bullet}^{m})$ by computing the requisite initial conditions; that is, we compute $H_3(C^{jf}_{\bullet}) = H_3(C^{j,j-1,1}_{\bullet})$ for $1 \leq j \leq 4$. Similarly, we compute $H_4(C^{m}_{\bullet})$ by computing $H_4(C_{\bullet}^{jf})$  for $1 \leq j \leq 5$. The former computation solves Conjecture 4.2 from \cite{PrWa2}; see Conjecture \ref{mainconjecture}.

\bigskip

We represent $\partial_4:C^{mf}_4 \to C^{mf}_3$ by the $\rank(C_3^{mf}) \times \rank(C_4^{mf})$ matrix with relations in columns and the basis as described in Section 3. We find the Smith Normal form\footnote{Notice that our ring, $k=\mathbb{Z}[y^2]$, is not a PID, and thus such a diagonalization (into Smith Normal form) is \textit{not} guaranteed. In fact, the diagonalization of boundary maps was performed by Mathematica over $\mathbb{Q}[y^2]$, not $\mathbb{Z}[y^2]$. However, the matrices encoding the row and column operations which diagonalize $\partial_n$ were verified as matrices over $\mathbb{Z}[y^2]$.} of this matrix over $k = \mathbb{Z}[y^2]$ and extract the torsion subgroups of $H_3(C_{\bullet}^{mf})$ from the diagonal and find $\rank \partial_4$. We then find  the matrix rank of the $\rank(C_2^{mf}) \times \rank(C_3^{mf})$ matrix representing $\partial_3: C^{mf}_3 \to C^{mf}_2$, which allows us to compute the free rank of $H_3(C^{mf}_\bullet)$ as $\rank(C_3^{mf}) - \rank \partial_4(C_4^{mf}) - \rank \partial_3(C_3^{mf})$. We outline the computation of $H_3(C_{\bullet}^{3f})$ as an example.

\begin{example}[Computing $H_3(C^{3f}_{\bullet})$]\label{computationm3n3} \ \\

    Let $m = 3$. We consider the portion of the chain complex 
    \[ C_4^{3f} \xrightarrow[]{\partial_4} C_3^{3f}  \xrightarrow[]{\partial_3} C_2^{3f}.\]

    A basis for $C_4^{3f}$ is the set of (classes of) $4$-tuples which each use the letters $2,3$ at least once. Explicitly, the basis has fifty elements, as follows:

    \begin{align*}
       & \{ (1, 1, 2, 3), (1, 1, 3, 2), (1, 2, 1, 3), (1, 2, 2, 3), (1, 2, 3, 
  1), (1, 2, 3, 2), (1, 2, 3, 3), (1, 3, 1, 2), \\ & (1, 3, 2, 1), (1, 3, 
  2, 2), 
   (1, 3, 2, 3), (1, 3, 3, 2), (2, 1, 1, 3), (2, 1, 2, 3), (2, 
  1, 3, 1), (2, 1, 3, 2), \\
  & (2, 1, 3, 3), (2, 2, 1, 3), (2, 2, 2, 
  3), (2, 2, 3, 1), (2, 2, 3, 2), (2, 2, 3, 3), (2, 3, 1, 1), (2, 3, 
  1, 2) \\ &
(2, 3, 1, 3), (2, 3, 2, 1), (2, 3, 2, 2), (2, 3, 2, 3), (2, 3, 3, 1), (2, 3, 3, 2), (2, 3, 3, 3),  (3, 1, 1, 2),  \\ & (3, 1, 2, 
  1), (3, 1, 2, 2), (3, 1, 2, 3), (3, 1, 3, 2), (3, 2, 1, 1), (3, 2, 
  1, 2), (3, 2, 1, 3), (3, 2, 2, 1), \\ &(3, 2, 2, 2), (3, 2, 2, 3), (3, 
  2, 3, 1),  (3, 2, 3, 2), (3, 2, 3, 3), (3, 3, 1, 2), (3, 3, 2, 
  1), (3, 3, 2, 2), \\ & (3, 3, 2, 3), (3, 3, 3, 2)
\}
    \end{align*}

 so $\rank C^{3f}_4 = 50.$  Similarly, a basis for $C_3^{3f}$ has the 12 following elements,
    \begin{align*}
         &  \{ (1, 2, 3), (1, 3, 2), (2, 1, 3), (2, 2, 3), (2, 3, 1), (2, 3, 2), (2, 3, 3), (3, 1, 2), \\ & \hspace{6 cm} (3, 2, 1), (3, 2, 2), (3, 2, 3), (3, 3, 2)  \};
    \end{align*}

so $\rank C_3^{3f} = 12$, and a basis for $C_2^{3f}$ has the two elements
    \[  \{ (2,3), (3,2) \} \] with $\rank C_3^{3f} = 2.$

\bigskip

We encode $\partial_3: C_3^{3;2,1} \to C_2^{3;2,1}$ as the $2 \times 12$ matrix with relations in columns. The rows and columns are indexed by the respective basis elements of $C_2^{3;2,1}$ and $C_3^{3;2,1}$ in lexicographic order. So, for instance, the first column says that $\partial((1,2,3)) = (1-y^2)(2,3) + y^2(1-y^2) (3,2).$ We diagonalize $\partial_3$ to find it has rank $1.$

    \[ \partial_3 = \left(
\begin{array}{cccccccccccc}
 1-y^2 & 1-y^2 & 0 & 1-y^2 & 0 & 0 & y^2-1 & 0 & 0 & 0 & 0 & 0 \\
 -y^2 \left(y^2-1\right) & y^2-y^4 & 0 & -y^2 \left(y^2-1\right) & 0 & 0 & y^2 \left(y^2-1\right) & 0 & 0 & 0 & 0 & 0 \\
\end{array}
\right)\]

\begin{align*}
    D_3 & =\left(
\begin{array}{cccccccccccc}
 y^2-1 & 0 & 0 & 0 & 0 & 0 & 0 & 0 & 0 & 0 & 0 & 0 \\
 0 & 0 & 0 & 0 & 0 & 0 & 0 & 0 & 0 & 0 & 0 & 0 \\
\end{array}
\right) \\
& = P_3 \partial_3 Q_3 \\
& = \left(
\begin{array}{cc}
 -1 & 0 \\
 -y^2 & 1 \\
\end{array}
\right) \partial_3 \left(
\begin{array}{cccccccccccc}
 1 & -1 & 0 & -1 & 0 & 0 & 1 & 0 & 0 & 0 & 0 & 0 \\
 0 & 1 & 0 & 0 & 0 & 0 & 0 & 0 & 0 & 0 & 0 & 0 \\
 0 & 0 & 1 & 0 & 0 & 0 & 0 & 0 & 0 & 0 & 0 & 0 \\
 0 & 0 & 0 & 1 & 0 & 0 & 0 & 0 & 0 & 0 & 0 & 0 \\
 0 & 0 & 0 & 0 & 1 & 0 & 0 & 0 & 0 & 0 & 0 & 0 \\
 0 & 0 & 0 & 0 & 0 & 1 & 0 & 0 & 0 & 0 & 0 & 0 \\
 0 & 0 & 0 & 0 & 0 & 0 & 1 & 0 & 0 & 0 & 0 & 0 \\
 0 & 0 & 0 & 0 & 0 & 0 & 0 & 1 & 0 & 0 & 0 & 0 \\
 0 & 0 & 0 & 0 & 0 & 0 & 0 & 0 & 1 & 0 & 0 & 0 \\
 0 & 0 & 0 & 0 & 0 & 0 & 0 & 0 & 0 & 1 & 0 & 0 \\
 0 & 0 & 0 & 0 & 0 & 0 & 0 & 0 & 0 & 0 & 1 & 0 \\
 0 & 0 & 0 & 0 & 0 & 0 & 0 & 0 & 0 & 0 & 0 & 1 \\
\end{array}
\right).
\end{align*}

We then diagonalize the $12 \times 50$ matrix representing $\partial_4: C_4^{3;2,1} \to C_3^{3;2,1}$ over $\mathbb{Z}[y^{ 2}]$. The matrices $P_4$ (encoding row operations) and $Q_4$ (encoding column operations) such that $D_4 = P_4 \partial_4 Q_4$ are shown in Figure \ref{P4form3} and Figure \ref{Q4form3} respectively.

\begin{figure}[H]
    \centering
    \includegraphics[width=0.75\linewidth]{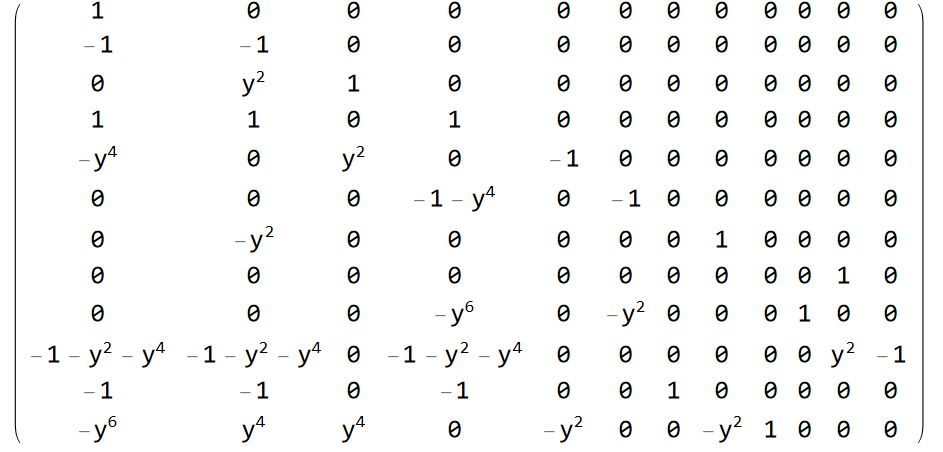}
    \caption{The row operation matrix $P_4$ \\ from the diagonalization of $\partial_4$.}
    \label{P4form3}
\end{figure}

\begin{figure}
    \centering
    \includegraphics[width=1\linewidth]{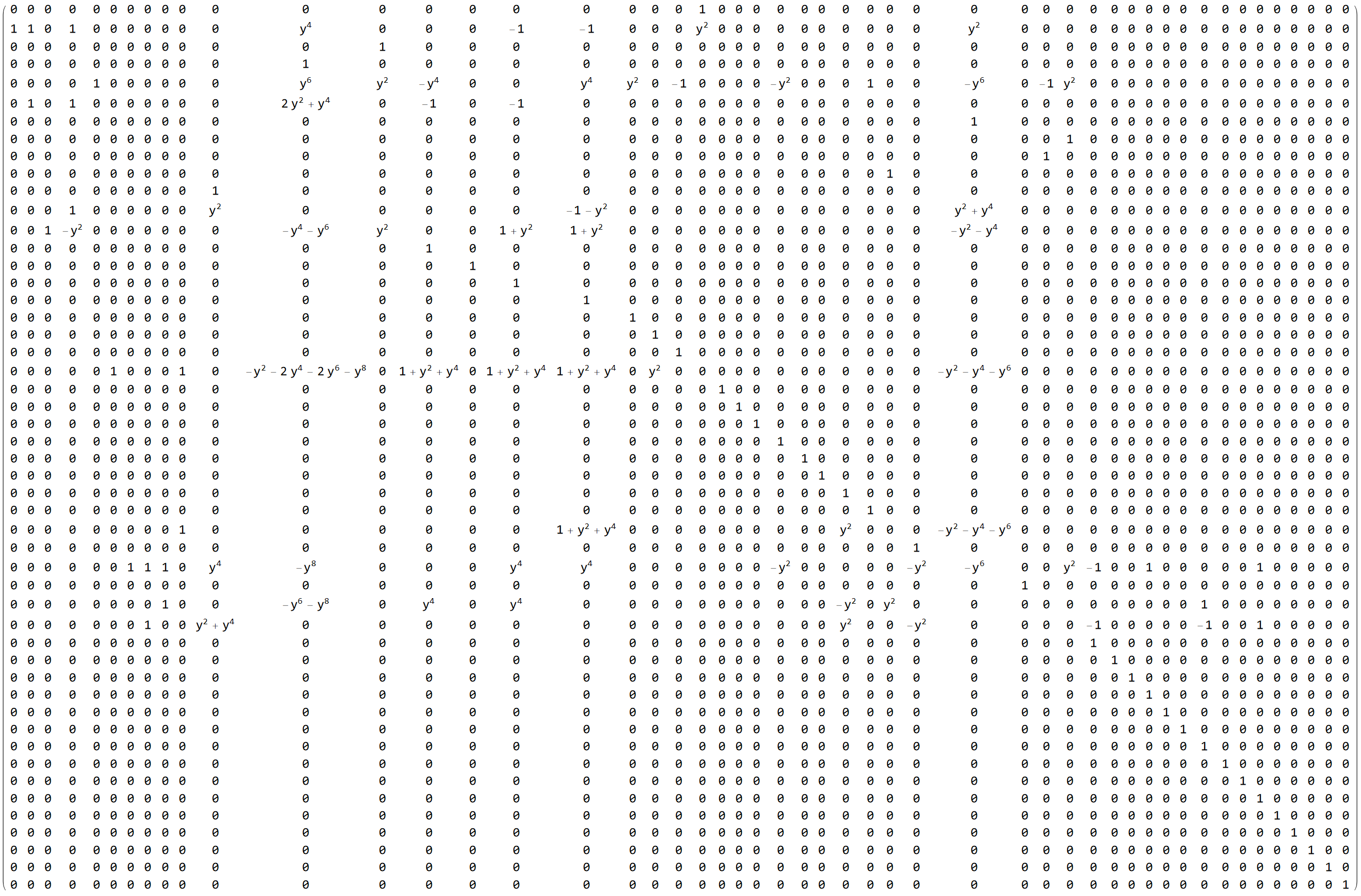}
    \caption{The column operation matrix $Q_4$ from the diagonalization of $\partial_4$.}
    \label{Q4form3}
\end{figure}

Finally, we can read the torsion subgroups of $H_3(C^{3f}_{\bullet})$ from the diagonal entries of $D_4$, which has diagonal $\{1 - y^2, 1 - y^2, 1 - y^2, 1 - y^2, 1 - y^2, 1 - y^2, 1 - y^2, 1 - y^2, 1 - y^4, 1 - y^4, 0, 0\}$, and compute the free rank as $\rank(C^{3f}_3) - \rank(\partial_4) - \rank(\partial_3)$ to conclude
\[ H_3(C_\bullet^{3f}) \cong k \oplus \left( \frac{k}{(1 - y^2)} \right)^{\oplus 8} \oplus \left( \frac{k}{(1 - y^4)} \right)^{\oplus 2}. \]



\end{example}

We recall that for $m>n+1$, $H_n(C_{\bullet}^{mf}) = 0$ since $C_n^{mf} = 0$. Thus, by following the example above and computing $H_3(C_{\bullet}^{mf})$ for $m=1,2,3,4$, we complete our computation of $H_3(C^{m}_{\bullet})$. The isomorphism type for $H_n(C_{\bullet}^{mf})$ is collected in the following table as triples $(a,b,c)$ where $H_n(C_{\bullet}^{mf}) \cong k^{\oplus a} \oplus \left( \frac{k}{(1 - y^2)} \right)^{\oplus b} \oplus \left( \frac{k}{(1 - y^4)} \right)^{\oplus c}$. 

\begin{remark}
    Note that the row and column operation matrices diagonalizing $\partial_n: C_n^{mf} \to C_{n-1}^{mf}$ for $1\leq m\leq 5$ and $1 \leq n \leq 5$ were too large to include in the text. They can be found at the following link: \url{https://tinyurl.com/PandQmatrices}.
\end{remark}

\begin{table}[H]
\begin{center}
\begin{tabular}{ c | c c c c c } 
   & $C_\bullet^1$ & $C_\bullet^{2f}$ & $C_\bullet^{3f}$ & $C_\bullet^{4f}$ & $C_\bullet^{5f}$  \\
   \hline 
   $H_1$ & (1,0,0) & (1,0,0) & \allzero & \allzero & \allzero \\
   $H_2$ & (1,0,0) & (1,1,1) & (1,1,0) & \allzero & \allzero  \\
   $H_3$ & (1,0,0) & (1,2,2) & (1,8,2) & (1,5,0) & \allzero \\
   $H_4$ & (1,0,0) & (1,6,4) & (1,33,6) & (1,51,3) & (1,23,0)  \\
\end{tabular}

\caption{Isomorphism types of $H_n(C^{mf}_{\bullet})$ for $1 \leq m \leq 5$ and $1 \leq n \leq 4$ \label{initialconditiondata}}
\end{center}

\end{table}

\begin{theorem}
    The third Yang-Baxter homology leading to the HOMFLYPT polynomial  is 
\begin{align*}
    H_3(C_\bullet^m) &\cong k^{\oplus \left( \binom{m-1}{0} + \binom{m-1}{1} + \binom{m - 1}{2} + \binom{m - 1}{3} \right)} \\
    &\oplus \left( \frac{k}{(1 - y^2)} \right)^{\oplus \left( 2 \binom{m - 1}{1} + 8\binom{m - 1}{2} + 5\binom{m - 1}{3} \right)} \\
    &\oplus \left( \frac{k}{(1 - y^4)} \right)^{\oplus \left(2\binom{m - 1}{1} + 2\binom{m - 1}{2} \right)} \\ 
   &  = k^{\frac{m(8-3m+m^2)}{6}}\oplus (k/(1-y^2))^{\frac{(m^2-1)(5m-6)}{6}}\oplus (k/(1-y^4))^{m(m-1)}
\end{align*}
\end{theorem}
\begin{proof}
    Use initial conditions from Table \ref{initialconditiondata} in conjunction with Proposition \ref{DirectSumResult}.
\end{proof}

Employing the same method and approaching the limits of our computation capacity, we arrive at a solution for the fourth homology as well.

\begin{theorem}
        The fourth Yang-Baxter homology leading to the HOMFLYPT polynomial is 
        \begin{align*} 
        H_4(R\m) &\cong k^{\oplus \left( \binom{m - 1}{0} + \binom{m - 1}{1} + \binom{m - 1}{2} + \binom{m - 1}{3} + \binom{m - 1}{4} \right) } \\
        &\oplus \left( \frac{k}{(1 - y^2)} \right)^{\oplus \left( 6 \binom{m - 1}{1} + 33\binom{m - 1}{2} + 51\binom{m - 1}{3} + 23\binom{m - 1}{4} \right)} \\
        &\oplus \left( \frac{k}{(1 - y^4)} \right)^{\oplus \left(4\binom{m - 1}{1} + 6\binom{m - 1}{2} + 3\binom{m - 1}{3} \right)} \\
        &= \medmath{ \left( \frac{m^4 - 6m^3 + 23m^2 - 18m + 24}{24}, \frac{(m-1)(23m^{3}-3m^{2}-26m+24)}{24}, \frac{(m-1)(m^2+m+2)}{2} \right) } \\
   \end{align*}
\end{theorem}

\subsection{Remark on $\rank(C_n^{(mf)})$}
We pause to compute $\rank(C_n^{mf})$ for given $m,n$, the usefulness of which is demonstrated in Example \ref{computationm3n3}. According to Milan Janjic in \cite{Jan}, we get the following.

\begin{proposition}
    Let $\tilde S(n,m,u)= \rank C_n^{m,u,m-u}$ that is the number of functions from \textcolor{red}{$[n]$} to $[m]$ such that the image of a function contains fixed subset of $[m]$ of size $u$. Then we have $\tilde S(n,m,u)= \tilde S(n,m-1,u) + \tilde S(n,m,u+1)$ (compare Proposition 3.6). Then 
$\tilde S(n,m,m)=m!S(n,m)$, where $S(n,m)$ is the Stirling number of the second kind and 
$$\tilde S(n,m,m-1) = \tilde S(n,m-1,m-1) + \tilde S(n,m,m) = (m-1)!S(n,m-1) + m!S(m,n)= (m-1)!S(n+1,m).$$ 
Furthermore, $\tilde S(n,m,u)$ satisfies inclusion-exclusion formula (a slight generalization of that for the Stirling numbers):
$$\tilde S(n,m,u)=\sum_{i=0}^{u} (-1)^i {u \choose i}(m-i)^n.$$
\end{proposition}

\begin{remark}
Table of values for $\tilde{S}(n,m,m-1).$ Notice that 
diagonal elements satisfy $\tilde{S}(m,m,m-1)=\frac{(m+1)!}{2}$  
 and the  subdiagonal satisfies $\tilde S(m,m+1,m)=m!$.

\begin{table}[H]
\begin{center}
\begin{tabular}{ c | c c c c c c  c}
m \textbackslash n & 1 & 2 & 3 & 4 & 5 & 6 & 7 \\ 
\hline 
1 & 1 & 1 & 1 & 1 & 1 & 1 & 1 \\
2 & 1 & 3 & 7 & 15 &31 &63 & 127\\
3 & 0 & 2 & 12 & 50 & 180 & 602 & 1932 \\
4 & 0 & 0 & 6 & 60 & 390 & 2100 & 10206 \\
5 & 0 & 0 & 0 & 24 &360 & 3360 & 25200\\
6 & 0 & 0 & 0 & 0 & 120 & 2520 & 31920 \\
7 & 0 & 0 & 0 & 0 & 0 & 720 & 20160 \\
\end{tabular}

\end{center}
\end{table}
\end{remark}




\section{Future work, speculations, and conjectures}

Both Yang-Baxter Homology and Khovanov-type homology generalize the Jones and HOMFLYPT polynomials. Naturally, then, a long-term goal is to connect these two theories. Our work in computing the third and fourth homology of the Yang-Baxter Operator leading to the Jones and HOMFLYPT polynomials is just the first step in gaining intuition in this direction. Another possible method of relating these two theories might be to find a geometric realization for Yang-Baxter homology to which we can compare Lipshitz-Sarkar's geometric realization of Khovanov homology  \cite{LiSa}.

In particular, we formulate the following conjectures about further calculations for the homology of $R_{(m)}$.



\begin{conjecture}[Free Rank of Yang Baxter]\label{free-rank}
    We conjecture that \[ \frank H_n(C_\bullet^{mf}) = \begin{cases}
        1 & m \leq n + 1 \\
        0 & m > n + 1.
    \end{cases} \] 
   
    As a corollary, we would then have that $ \frank H_n(C_\bullet^m)$ is \[ \sum_{k = 1}^m \binom{m - 1}{k - 1} \frank H_n(C_\bullet^{mf}) = \sum_{k = 1}^{n + 1} \binom{m - 1}{k - 1} = \sum_{k = 0}^n \binom{m - 1}{k} \] when $m \geq n + 1$ and \[ \sum_{k = 1}^m \binom{m - 1}{k - 1} = \sum_{k = 0}^{m - 1} \binom{m - 1}{k} = 2^{m - 1} \] when $m \leq n + 1$. 
\end{conjecture}



    

\begin{conjecture}[$H_n(R_{(2)})$]
   We conjecture about the homology of $R_{(2)}$ as follows:
    \[ H_n(R_{(2)}) = k^2 \oplus \left( \frac{k}{(1 - y^2)} 
    \right)^{\oplus b_n(2)} \oplus \left( \frac{k}{(1 - y^4)} 
     \right)^{\oplus c_n(2)} \] where

    $c_n(2)=f_{n+1}-1$ and $f_n$ are  the Fibonacci numbers ($f_0 = 0$, $f_1 = 1$), and \[ b_n(2) = \frac{2^{n + 1} + (-1)^n}{3} - f_{n + 1}. \] 

    Notice that $H_n(R_{(2)})=k\oplus H_n(C_\bullet^{2f}).$
    
\end{conjecture}
\ \\ \ \\

Recall that by Theorem \ref{DirectSumResult},  \begin{align*}
    H_5(C_\bullet^m) & = H_5(C_\bullet^{1,0}) \oplus (m - 1)H_5(C_\bullet^{2,1}) \oplus \binom{m - 1}{2} H_5(C_\bullet^{3,2}) \oplus \\ & \binom{m - 1}{3} H_5(C_\bullet^{4,3}) \oplus \binom{m - 1}{4} H_5(C_\bullet^{5,4}) \oplus \binom{m - 1}{5} H_5(C_\bullet^{6,5}).
\end{align*}   Accordingly, we compute (over $\Z$ with $y = 2$): \begin{align*}
    H_5(C_\bullet^{1,0}) & = \Z \\
    H_5(C_\bullet^{2,1}) & = \Z \oplus 13 \Z_3 \oplus 7 \Z_{15} \\
    H_5(C_\bullet^{3,2}) &= \Z \oplus 124 \Z_3 \oplus 16 \Z_{15} \\
    H_5(C_\bullet^{4,3}) &= \Z \oplus 323 \Z_3 \oplus 12 \Z_{15} \\
    H_5(C_\bullet^{5,4}) &= \Z \oplus 332 \Z_3 \oplus 4 \Z_{15} \\
    H_5(C_\bullet^{6,5}) &= \Z \oplus 119 \Z_3.
\end{align*}

These leads to the conjecture (for $k=Z[y^2]$):
\begin{conjecture}[Conjecture for $H_5(R_{(m)}$]
The fifth Yang-Baxter homology leading to the HOMFLYPT polynomial is 
        \begin{align*} 
        H_5(R\m) &\cong k^{\oplus \left( \binom{m - 1}{0} + \binom{m - 1}{1} + \binom{m - 1}{2} + \binom{m - 1}{3} + \binom{m - 1}{4} + \binom{m - 1}{5} \right) } \\
        &\oplus \left( \frac{k}{(1 - y^2)} \right)^{\oplus \left( 13 \binom{m - 1}{1} + 124\binom{m - 1}{2} + 323\binom{m - 1}{3} + 332\binom{m - 1}{4} + 119\binom{m - 1}{5} \right)} \\
        &\oplus \left( \frac{k}{(1 - y^4)} \right)^{\oplus \left(7\binom{m - 1}{1} + 16\binom{m - 1}{2} + 12\binom{m - 1}{3} + 4\binom{m - 1}{4} \right)} \\
        &= \medmath{ \bigg ( \frac{m\left(m^{4}-10m^{3}+55m^{2}-110m+184\right)}{120}, \frac{\left(m^{2}-1\right)\left(119m^{3}-125m^{2}+94m-120\right)}{120},} \\
        &\qquad \qquad \qquad \qquad \medmath{ \frac{\left(m-1\right)\left(m^{3}+3m^{2}+14m-6\right)}{6} \bigg ) } 
   \end{align*}
\end{conjecture}



One possible approach to these computational conjectures is through filtration of the chain modules $C_n^{m}$ as in this paper. We present some possible tools for finding further filtrations in the following subsections: in \ref{Kunneth}, we present a K\"unneth-type formula, and in \ref{specdecomp} we present some conjectured splittings of chain complexes evidenced by experimental data.


\subsection{K\"unneth subchain complex of $(C_n^{m},\partial_n)$}\label{Kunneth}
There is an intriguing subchain complex of $(C_n^{m},\partial_n)$ which is isomorphic to the tensor product of smaller chain complexes (this allows us to use a K\"unneth formula, as long as we work in a PID; say $\Q[y^2]$-ring). We conjecture that this subcomplex splits from $(C_n^{m},\partial_n)$. 
This subcomplex, denoted by $(C_n^{(m,B,A)},\partial_n)$ is constructed as follows:
\begin{definition}\label{KunethDef}
Fix a decomposition of letters $X_{(m)}=A\cup B$ 
and consider a submodule $C_n^{(m,A,B)}$ of $C_n^{m}$ 
generated by sequences in $X^n_{(m)}$ which start with letters from $A$ and end with letters in $B$. More formally:
$$X_{(n,m,A,B)} =\{(a_1,a_2,...,a_n)\ | \ \mbox{ there is $i$ such that } a_1,...,a_i \in A; a_{i+1},...,a_n \in B \}.$$
\end{definition}
\begin{proposition}
If $A \geq B$; that is if $v_i\in A$ and $v_j\in B$ then 
 $i\geq j$ then 
 \begin{enumerate} \item[(1)]
 $(C_n^{(m,A,B)},\partial_n)$ is a subchain complex of $(C_n^{m},\partial_n)$. 
  \item[(2)]
 If $(a_1,a_2,...,a_n)=(u,w)$ where $u\in A^{|u|}$, $w\in B^{|w|}$ and $|x| $ denotes the length of the word $x$ then
 $$\partial_n(u,w)= (\partial_{|u|} u,w) + (-1)^{|u|}(u,\partial_{|w|}w).$$.
 \item[(3)] If $A>B$ then $C_{\bullet}^{(m,A,B)} \equiv C_{\bullet}^{|A|}\otimes C_{\bullet}^{|B|},$
 that is $C_n^{(m,A,B)}= \sum_{i+j=n}C_i^{|A|} \otimes C_j^{|B|}$.
 \end{enumerate}
\end{proposition}
\begin{proof} (1) and (3) follow from (2). To show (2) we 
use the fact than if in $(a_1,...,a_n)$  we have $a_i\geq a_j$ for some $j$ then $d_j^{\ell}=d_{j+1}^{\ell}$ and 
$d_j^{r}=d_{j+1}^{r}$.
\end{proof}
\begin{conjecture} The short exact sequence of chain complexes
$$0 \to C_{\bullet}^{(m,A,B)} \to C_\bullet^{m} \to C_n^{m}/C_{\bullet}^{(m,A,B)} \to 0,$$
splits. Thus homology $H_n(C_\bullet^{m})$ splits.
\end{conjecture}
The special cases of the conjecture, when $A=\{v_m\}$ and 
$B=X_{(m)}$ or $B=X_{(m-1)}$  are also of interest:
\begin{conjecture} 
\begin{enumerate}
\item[(i)] The subcomplex $v_mC_{n-1}^{(m)}$ of $C_n^{m}$ splits $C_n^{m}$.
\item[(ii)] The subcomplex $v_mC_{n-1}^{(m-1)}$ of $C_n^{m}$ splits $C_n^{m}$.
\end{enumerate}
\end{conjecture}

\subsection{Additional torsion in $H_n(R_{(m)})$.} \label{specdecomp}


\vspace{1em}

Initially it was tempting to guess that for every $n$, $H_n(R_{(m)})$ decomposes into cyclic parts
$$H_n(R_{(m)}) = k^{a_n(m)} \oplus \bigg(k/(1-y^2)\bigg)^{b_n(m)} \oplus \bigg(k/(1-y^4)\bigg)^{c_n(m)}.$$
where $a_n(m)$ is described in Conjecture \ref{free-rank} and $b_n(m)$ and $c_n(m)$ are polynomials in variable $m$ of degree $n$ and $n-1$ respectively. However, the computation of $H_6(C^{3f})$ over $\mathbb Z$ ($y=2$) suggests another torsion:
$k/(1-y^2)(1+y^2)(1+y^2+y^4)$.  This suggests the following conjecture.

\begin{conjecture} \begin{enumerate} \item[(1)]
Over the ring $k=\mathbb Z$ ($y=2$) we have
$$H_6(R_{(m)},\mathbb Z)= \Z^{a_6(m)}\oplus \Z_{3}^{b_6(m)} \oplus \Z_{15}^{c_6(m)} \oplus \Z_{315}^{d_6(m)}, $$
with $a_6(m)$ as described in Conjecture \ref{free-rank}.

\item[(2)] For the ring $k=\Z[y^2]$ we have:
$$H_6(R_{(m)})= k^{a_6(m)} \oplus \bigg(k/(1-y^2)\bigg)^{b_n(m)} \oplus \bigg(k/(1-y^4)\bigg)^{c_6(m)}\oplus \bigg(k/(1-y^2)(1+y^2)(1+y^2+y^4)\bigg)^{d_6(m)}. $$
\end{enumerate}
\end{conjecture}

We show that for $d_6(m)$ in part (1), $d_6(m) \geq m(m-1)$. See Table \ref{conjsplitdata}.

We can try guessing more; motivated by symmetry, but without much evidence we would suggest the following conjectures:
\begin{enumerate}
\item[(1)] $H_n(R_{(m)})$ can be decomposed as a sum of cyclic $k$-modules.
\item[(2)] Each factor is of type $k/(1-y^2)[n]_{y^2}!$ where 
$[n]_{y^2}=1+y^2+y^4+...+y^{2n-2}$ and $[n]_{y^2}!=[1]_{y^2}[2]_{y^2}[n]_{y^2}.$
\item[(3)] With much less confidence we would guess that 
the lowest index of homology having torsion $k/(1-y^2)[n]_{y^2}!$ is $H_{n!}(R_{(m)})$. 
Despite the fact that (3) has not much experimental support we think it is worth of exploring.
\end{enumerate}

\subsection{Exploring splittings of $C_{\bullet}^{mf}$}

In addition to the K\"unneth subchain complex of $C_\bullet^m$, we have sought to further decompose the chain complexes $C_\bullet^{mf}$. To this end, we define the $C_\bullet^{mf,\ell}$ to be the subchain complex of $C_\bullet^{mf}$ whose chain modules have basis composed of tuples using the top letter $v_m$ of the alphabet at most $\ell$ times (and using $v_2, \dots, v_m$ at least once). We then have the short exact sequences of chain complexes \[ \begin{tikzcd}
	0 \arrow[r] & C_\bullet^{mf,\ell} \arrow[r] & C_\bullet^{mf} \arrow[r] & C_\bullet^{mf} / C_\bullet^{mf, \ell} \arrow[r] & 0,
\end{tikzcd} \] and the  computational data in Table \ref{conjsplitdata} lead us to  conjecture the above sequence splits and, therefore, its related homology. 

\begin{conjecture}\label{conjsplitcmfl} Over the ring $k = \mathbb{Z}[y^{2}],$

    \[ H_n(C_{\bullet}^{mf}) = H_n(C_{\bullet}^{mf,\ell}) \oplus H_n(C_{\bullet}^{mf}/C_{\bullet}^{mf,\ell}). \]
\end{conjecture}

\vspace{1em}

\begin{table}[H]
\begin{center}
\begin{tabular}{ c | c c c }
    $n,m,\ell$ & $H_n(C_\bullet^{mf,\ell})$ & $H_n(C_\bullet^{mf})$ & $H_n(C_\bullet^{mf} / C_\bullet^{mf,\ell})$ \\
    \hline 
    $4,2,1$ & $(1,2,0)$ & $(1,6,4)$ & $(0,4,4)$ \\
    $4,2,2$ & $(1,4,2)$ & $(1,6,4)$ & $(0,2,2)$ \\
    $4,3,1$ & $(1,18,2)$ & $(1,33,6)$ & $(0,15,4)$ \\
    $4,3,2$ & $(1,30,5)$ & $(1,33,6)$ & $(0,3,1)$ \\
    $5,2,1$ & $(1,2,0)$ & $(1,13,7)$ & $(0,11,7)$ \\
    $5,2,2$ & $(1,6,2)$ & $(1,13,7)$ & $(0,7,5)$ \\
    $5,2,3$ & $(1,11,4)$ & $(1,13,7)$ & $(0,2,3)$ \\
    $5,3,1$ & $(1,50,4)$ & $(1,124,16)$ & $(0,74,12)$ \\
    $5,3,2$ & $(1,97,12)$ & $(1,124,16)$ & $(0,27,4)$ \\
    $5,3,3$ & $(1,120,15)$ & $(1,124,16)$ & $(0,4,1)$ \\
    $5,3,1$ & $(1,50,4)$ & $(1,124,16)$ &  $(0,74,12)$ \\
    $5,3,2$ & $(1,97,4)$ & $(1,124,16)$ &  $(0,27,4)$ \\
    $5,4,1$ & $(1,201, 6)$ & $(1,323,12)$ & $(0,122,6)$  \\
    $5,4,2$ & $(1,304,11)$ & $(1,323,12)$ & $(0,19,1)$ \\
    $5,4,3$ & $(1,323,12)$ & $(1,323,12)$ & $(0,0,0)$ \\
    $6,2,1$ & $(1,3,0)$ & $(1,30,12)$ & $(0,27,12)$ \\
    $6,2,2$ & $(1,9,3)$ & $(1,30,12)$ & $(0,21,9)$ \\
    $6,2,3$ & $(1,20,5)$ & $(1,30,12)$ & $(0,10,7)$ \\
    $6,2,4$ & $(1,27,9)$ & $(1,30,12)$ & $(0,3,3)$ \\
    $6,3,1$ & $(1,124,7,0)$ & $(1,423,36,2)$ & $(0,299,29,2)$ \\
    $6,3,2$ & $(1,277,24,0)$ & $(1,423,36,2)$ & $(0,146,12,2)$ \\
\end{tabular}
\end{center}
\caption{These homology groups are calculated with $y=2$ substitution, and suggest splitness of $C_{\bullet}^{mf}$. In the case of $(6,3,1)$ and $(6,3,2)$, additional torsion is confirmed for $y=3$ substitution, as well.}
\end{table} 
\label{conjsplitdata}
\vspace{1em}

\color{black}

\subsection{More structure?}

We conclude with two further possibilities for studying Yang-Baxter homology. First, we can use the involution $\tau$ from Section~\ref{sec:DS} to introduce $\mathbb{Z}_2$-equivariant homology of $R_{(m)}$ (see e.g. \cite{Tu,BPS}). Second, when constructing the homology of Yang-Baxter operators $R_{(m)}$ we naturally start from precubic modules. On the other hand precubic sets have natural geometric realizations (elements of $X_n$ are the names of $n$-cubes and $d_{i,n}^{\varepsilon}$ are gluing instructions). In fact, the homology theory for racks and quandles  began with this geometric realization, see Roger Fenn's paper \cite{Fenn}. The geometric realization of the
homology of $R_{(m)}$ seems to be more difficult, and reminds us of the discovery of the geometric realization of Khovanov homology, \cite{LiSa}. Furthermore, we plan to analyze degenerate maps $s_{i,n}: C_n(R_{(m)}) \to C_{n+1}(R_{(m)}) $ given by $s_{i,n}(a_1,...,a_n)=(a_1,...,a_{i-1},a_i,a_i,a_{i+1},...,a_n)$, and look for an analogue of the (very) weak simplicial modules introduced in \cite{Prz-5,Prz-8}.

\newpage

\section*{Acknowledgements}
We would like to thank Seung Yeop Yang for  valuable discussions. The fourth and seventh authors were supported by the LAMP Program of the National Research Foundation of Korea, grant No. RS-2023-00301914, and by the MSIT grant No. 2022R1A5A1033624. The fifth author was partially supported by the Simons Collaboration Grant 637794. The sixth author was partially supported by the National Natural Science Foundation of China (Grant No. 11901229, Grant No. 12371029, Grant No. 22341304).

\end{document}